\documentclass[12pt,leqno,twoside,sort]{amsart}

\usepackage{amssymb,amsmath,amsthm,soul,color}

\usepackage[normalem]{ulem}

\usepackage{amssymb,amsmath,amsthm}

\usepackage{mathrsfs}

\usepackage[utf8]{inputenc}

\usepackage{todonotes}

\usepackage[left=2cm, right=2cm, top=2cm, bottom=2cm]{geometry}

\usepackage{url}

\usepackage[numbers,sort&compress]{natbib}

\usepackage{fixmath}

\usepackage[T1]{fontenc}
\usepackage{lmodern}
\usepackage{microtype}
\usepackage[normalem]{ulem}

\usepackage[colorlinks=true,urlcolor=blue,
citecolor=red,linkcolor=blue,linktocpage,pdfpagelabels,
bookmarksnumbered,bookmarksopen]{hyperref}

\usepackage{xcolor,cancel}

\newtheorem{Th}{Theorem}[section]
\newtheorem{Prop}[Th]{Proposition}
\newtheorem{Lem}[Th]{Lemma}

\newcommand{\R}{\mathbb{R}}

\newcommand{\cA}{{\mathcal A}}

\newcommand{\cC}{{\mathcal C}}
\newcommand{\cD}{{\mathcal D}}

\newcommand{\cI}{{\mathcal I}}
\newcommand{\cJ}{{\mathcal J}}

\newcommand{\cN}{{\mathcal N}}

\newcommand{\cS}{{\mathcal S}}

\newcommand{\Ga}{\Gamma}
\newcommand{\Om}{\Omega}

\newcommand{\dist}{\mathrm{dist}}

\newcommand{\weakto}{\rightharpoonup}

\numberwithin{equation}{section}

\newcommand{\supp}{\mathrm{supp}\,}

\begin{document}

\title{Elliptic problems with mixed nonlinearities and potentials singular at the origin and at the boundary of the domain}

\author[B. Bieganowski]{Bartosz Bieganowski}
	\address[B. Bieganowski]{\newline\indent
			Faculty of Mathematics, Informatics and Mechanics, \newline\indent
			University of Warsaw, \newline\indent
			ul. Banacha 2, 02-097 Warsaw, Poland}	
			\email{\href{mailto:bartoszb@mimuw.edu.pl}{bartoszb@mimuw.edu.pl}}	

\author[A. Konysz]{Adam Konysz}
\address[A. Konysz]{\newline\indent  	Faculty of Mathematics and Computer Science,		\newline\indent Nicolaus Copernicus University, \newline\indent ul. Chopina 12/18, 87-100 Toru\'n, Poland}
\email{\href{mailto:adamkon@mat.umk.pl}{adamkon@mat.umk.pl}}

\date{}	\date{\today}

\begin{abstract} 
We are interested in the following Dirichlet problem
$$
\left\{ \begin{array}{ll}
-\Delta u + \lambda u - \mu \frac{u}{|x|^2} - \nu \frac{u}{\dist(x,\R^N \setminus \Omega)^2} = f(x,u)  & \quad \mbox{in } \Omega \\
u = 0 & \quad \mbox{on } \partial \Omega,
\end{array} \right.
$$
on a bounded domain $\Omega \subset \R^N$ with $0 \in \Omega$. We assume that the nonlinear part is superlinear on some closed subset $K \subset \Omega$ and asymptotically linear on $\Omega \setminus K$. We find a
solution with the energy bounded by a certain min-max level, and infinitely many solutions provided that $f$ is odd in $u$. Moreover we study also the multiplicity of solutions to the associated normalized problem.

\medskip

   
\end{abstract}

\maketitle

\pagestyle{myheadings} \markboth{\underline{B. Bieganowski, A. Konysz}}{
		\underline{Elliptic problems with singularities and mixed nonlinearities}}

\section{Introduction}

We are interested in the problem
\begin{equation}\label{eq}
\left\{ \begin{array}{ll}
-\Delta u + \lambda u - \mu \frac{u}{|x|^2} - \nu \frac{u}{\dist(x,\R^N \setminus \Omega)^2} = f(x,u)  & \quad \mbox{in } \Omega \\
u = 0 & \quad \mbox{on } \partial \Omega,
\end{array} \right.
\end{equation}
where $\lambda, \mu \in \R$ are real parameters, $f : \Omega \times \R \rightarrow \R$, $\Omega \subset \R^N$ is a bounded domain in $\R^N$ with $0 \in \Omega$, and $K \subset \Omega$ is a closed set with $| \mathrm{int}\, K | > 0$.

Semilinear problems of general form
$$
-\Delta u = h(x,u)
$$
appear when one looks for stationary states of time-dependent problems, including the \textit{heat equation} $\frac{\partial u}{\partial t} - \Delta u = h(x,u)$ or the \textit{wave equation} $\frac{\partial^2 u}{ \partial t^2} - \Delta u = h(x,u)$. In nonlinear optics the \textit{nonlinear Schr\"odinger equation} is studied
\begin{equation}\label{eq:schroed}
\mathbf{i} \frac{\partial \Psi}{\partial t} + \Delta \Psi = h(x, |\Psi|) \Psi, \quad (t,x) \in \R \times \Omega
\end{equation}
and looking for standing waves $\Psi(t,x) = e^{i\lambda t} u(x)$ leads then to a semilinear problem.

The time-dependent equation \eqref{eq:schroed} appears in physical models in the case of bounded domains $\Omega$ (\cite{Fibich, Fibich2, Zuazua}), as well as in the case $\Omega = \R^N$ (\cite{Dorfler, Nie}). Two points of view of solutions to \eqref{eq} are possible; either $\lambda$ may be prescribed or may be considered as a part of the unknown. In the latter case a natural additional condition is the prescribed mass $\int_\Omega u^2 \, dx$. In the paper we will consider both cases, namely we will look for solutions for the unconstrained problem \eqref{eq} as well as the constrained one, see \eqref{eq:normalized} below. 

The equation \eqref{eq} (and systems of such equations) on bounded domains has been studied in the presence of bounded potentials \cite{B2} and singular at the origin \cite{Gao}, see also \cite{Felli, GuoMed, Kostenko} for the case  of unbounded $\Omega$. Its constrained counterpart without the potential has been studied e.g. in \cite{Noris,PV}, where \eqref{eq:normalized} was studied with $f(x,u)=|u|^{p-2}u$, $\nu=\mu=0$ in the mass-subcritical, mass-critical and mass-supercritical cases. In this paper we are interested in the presence of a potential 
$$
V(x) = -\frac{\mu}{|x|^2} - \frac{\nu}{\mathrm{dist}\,(x, \mathbb{R}^N \setminus \Omega)^2}
$$
which is singular in $\Omega$ as well as on the whole boundary $\partial \Omega$. We mention here that Schr\"odinger operators were studied with potentials being singular at the point on the boundary \cite{Chen}, as well as with potentials being singular on the whole boundary \cite{Tai, Tai2}.
We assume that $\Omega$ is a domain satisfying the following condition
\begin{enumerate}
\item[(C)] $-\Delta d \geq 0$ in $\Omega$, in the sense of distributions, where $d(x) := \dist (x, \R^N \setminus \Omega)$.
\end{enumerate}
This condition allows us to study the singular potential by means of Hardy-type inequalities (Section \ref{sect:2}). As we will see in Section \ref{sect:2} (see Proposition \ref{prop:convex}), any convex domain $\Omega$ satisfies (C).

We impose the following condition on parameters appearing in the problem
\begin{enumerate}
\item[(N)] $\mu,\nu \geq 0$, $\frac{\mu}{(N-2)^2} +  \nu < \frac{1}{4}$, $N \geq 3$.
\end{enumerate}

On the nonlinear part of \eqref{eq} we propose the following assumptions.

\begin{enumerate}
\item[(F1)] $f : \Omega \times \R \rightarrow \R$ is a Carath\'eodory function and there is $2 < p < 2^*$ such that
$$
|f(x,u)| \lesssim 1+|u|^{p-1}, \quad u \in \R, \ x \in \Omega
$$
\item[(F2)] $f(x,u) = o(u)$ uniformly in $x \in \Omega$ as $u \to 0$
\item[(F3)] $F(x,u)/|u|^2 \to +\infty$ as $|u| \to +\infty$ for $x \in K$, where $F(x,u) := \int_0^u f(x,s) \, ds$;
\item[(F4)] $f(x,u)/|u|$ is nondecreasing on $(-\infty, 0)$ and on $(0,\infty)$
\item[(F5)] $f(x,u)=\Theta(x)u$ for sufficiently large $|u|$ and $x \in \Omega \setminus K$, where $\Theta \in L^\infty(\Omega \setminus K)$.
\end{enumerate}

A simple example satisfying all foregoing conditions is the following
$$
f(x,u) = \left\{ \begin{array}{ll}
\Gamma(x) |u|^{p-2}u,     & \quad x \in K, \\
\frac{|u|^2}{1+|u|^2} u \chi_{|u| \leq 1} + \frac12 u \chi_{|u| > 1}    & \quad x \in \Omega \setminus K,
\end{array} \right.
$$
where $\Gamma\in L^\infty(K)$ is a nonnegative function.

To show the boundedness of minimizing sequences to the problem \eqref{eq} we impose the following abstract condition
\begin{enumerate}
\item[(A)] $-\lambda$ is not an eigenvalue of $-\Delta - \frac{\mu}{|x|^2} - \frac{\nu}{d(x)^2} - \Theta(x)$ with Dirichlet boundary conditions on $L^2(\Omega \setminus K)$.
\end{enumerate}
As we will see in Section \ref{sect:2}, (A) is satisfied if e.g. $\lambda \geq |\Theta|_\infty$ (cf. Theorem \ref{th:spectr}).

\begin{Th}\label{th:main1}
Suppose that (C), (N), (F1)--(F5), (A) are satisfied and $\lambda \geq 0$. Then there is a nontrivial weak solution to \eqref{eq} with the energy level $c$ satisfying \eqref{gsl}. 
\end{Th}

\begin{Th}\label{th:main2}
Suppose that (C), (N), (F1)--(F5), (A) hold, $\lambda \geq 0$, and $f$ is odd in $u \in \R$. Then there is infinitely many weak solutions to \eqref{eq}.
\end{Th}

In the last section we also study the normalized problem
\begin{equation}\label{eq:normalized}
\left\{ \begin{array}{ll}
-\Delta u + \lambda u - \mu \frac{u}{|x|^2} - \nu \frac{u}{\dist(x, \R^N\setminus\Omega)^2} = f(x,u) & \quad \mbox{in } \Omega \\
u = 0 & \quad \mbox{on } \partial \Omega, \\
\int_\Omega u^2 \, dx = \rho > 0,
\end{array} \right.
\end{equation}
where $\rho$ is fixed and $(\lambda, u) \in \R \times H^1_0(\Omega)$ is an unknown. Then we obtain the following multiplicity result in the so-called \textit{mass-subcritical case}.

\begin{Th}\label{main:3}
Suppose that (C), (N), (F1) hold with $p < 2_* := 2+\frac{4}{N}$, and $f$ is odd in $u \in \R$. Then there is infinitely many weak solutions to \eqref{eq:normalized}.
\end{Th}

In what follows, $\lesssim$ denotes the inequality up to a multiplicative constant. Moreover $C$ denotes a generic constant which may vary from one line to another.

\section{The domain \texorpdfstring{$\Omega$}{Ω} and the singular Schr\"odinger operator}\label{sect:2}

We recall that if $A \subset \R^N$ is a closed, nonempty set, we can define the distance function $\dist (\cdot, A) : \R^N \rightarrow [0,\infty)$ by
$$
\dist(x, A) := \inf_{y \in A} |x-y|, \quad x \in \R^N.
$$
We collect the following properties of the distance function:
\begin{itemize}
    \item[(i)] $|\dist(x, A) - \dist(x',A) \leq |x-x'|$ for all $x,x' \in \R^N$,
    \item[(ii)] if $x \in \R^N \setminus A$, then $\dist(\cdot, A)$ is differentiable at $x$ if and only if there is unique $u \in A$ such that $\dist(x,A) = |x-y|$ and then $\nabla \dist(x, A) = \frac{x-y}{|x-y|}$.
\end{itemize}
Now we consider $A := \R^N \setminus \Omega$. It is clear that $A$ is a closed subset of $\R^N$. We recall that we denote $d(x) = \dist (x, \R^N \setminus \Omega)$. Observe that, due to Rademacher’s theorem (\cite[Theorem 2.2.1]{Ziemer}) and (i), $d$ is differentiable almost everywhere and, from (ii) $|\nabla d| = 1$, almost everywhere. We remind that the assumption (C) says that
$$
-\Delta d \geq 0 \mbox{ in } \Omega
$$
holds in the sense of distributions. We note the following fact.

\begin{Prop}\label{prop:convex}
If $\Omega \subset \R^N$ is a convex domain in $\R^N$, then $d \big|_\Omega$ is concave and satisfies (C).
\end{Prop}

\begin{proof}
First note that $d \big|_\Omega : \Omega \rightarrow [0,\infty)$ is a concave function. Indeed, fix any $x,y \in \Omega$ and $\alpha \in [0,1]$. Let $z = \alpha x + (1-\alpha)y$. Choose $z_0 \in \partial \Omega$ such that $d(z) = |z-z_0|$. Let $T_{z_0} := z_0 + \mathrm{span} \{ z-z_0 \}^\perp$ be an affine subspace of $\R^N$  orthogonal to $z-z_0$ containing $z_0$. Define $x_0, y_0$ as orthogonal projections of $x, y$ respectively onto $T_{z_0}$. Then
$$
d(z) = |z-z_0| = \alpha |x-x_0| + (1-\alpha)|y-y_0| \geq \alpha d(x) + (1-\alpha) d(y),
$$
which completes the proof of concavity. Moreover, since $d$ is concave on $\Omega$, from \cite[Theorem 6.8]{Ev} there is a nonnegative Radon measure $\mu$ on $\Omega$ satisfying
$$
-\Delta d = \mu \quad \mbox{in the sense of distributions, }
$$
namely
$$
\int_{\Omega} \nabla d \cdot \nabla \varphi \, dx = \int_\Omega \varphi \, d\mu \quad \mbox{for } \varphi \in \cC_0^\infty (\Omega).
$$
Clearly, for $\varphi \geq 0$ we get
$$
\int_{\Omega} \nabla d \cdot \nabla \varphi \, dx \geq 0
$$
and condition (C) holds.
\end{proof}
To study singular terms in \eqref{eq} we recall the following Hardy-type inequalities. If $u \in H^1_0 (\Omega)$, where $\Omega$ is a domain in $\R^N$ with finite Lebesgue measure and $0 \in \Omega$, then (see \cite{BV})
\begin{equation}\label{hardy1}
    \frac{(N-2)^2}{4} \int_\Omega \frac{u^2}{|x|^2} \, dx \leq \int_\Omega |\nabla u|^2 \, dx.
\end{equation}
Now let $\Omega \subset \R^N$ be a bounded domain satisfying (C). Then, for $u \in H^1_0(\Omega)$, the following Hardy inequality involving the distance function holds (see \cite{BFT})
\begin{equation}\label{hardy2}
\frac14 \int_\Omega \frac{u^2}{d(x)^2} \, dx \leq \int_{\Omega} |\nabla u|^2 \, dx.
\end{equation}
We consider the operator $\mathcal{A} := -\Delta - \frac{\mu}{|x|^2} - \frac{\nu}{d(x)^2} - \Theta(x)$ on $L^2 (\Omega \setminus K)$ with Dirichlet boundary conditions. Then the domain is $\cD(\cA) := H^2 (\Omega \setminus K) \cap H^1_0 (\Omega \setminus K)$.

\begin{Th}\label{th:spectr}
The operator $\cA : \cD(A) \subset L^2 (\Omega \setminus K) \rightarrow L^2 (\Omega \setminus K)$ is elliptic, self-adjoint on $L^2 (\Omega \setminus K)$ and has compact resolvents. Moreover the spectrum $\sigma(\cA) \subset (-|\Theta|_\infty, +\infty)$ and consists of eigenvalues $-|\Theta|_\infty < \lambda_1 \leq \lambda_2 \leq \ldots$ with $\lambda_n \to +\infty$ as $n \to +\infty$.
\end{Th}

\begin{proof}
It is well-known that $\cD(\cA)$ is closed in $L^2(\Omega \setminus K)$. It is easy to check that $\cA$ is self-adjoint. The compactness of its resolvents easily follows from the Rellich-Kondrachov theorem. Hence its spectrum consists of eigenvalues $\lambda_n$ with $\lambda_n \to +\infty$. To see that $\sigma(\cA) \subset (-|\Theta|_\infty, +\infty)$, suppose that $\lambda$ is an eigenvalue of $\cA$ with an associated eigenfunction $u \in H^1_0(\Omega \setminus K)$. We can treat $u$ as a function in $H^1_0(\Omega)$ being zero on $K$. Then, using \eqref{hardy1} and \eqref{hardy2},
\begin{align*}
\lambda \int_{\Omega \setminus K} u^2 \, dx &\geq - \int_\Omega |\nabla u|^2 + \mu \frac{u^2}{|x|^2} + \nu \frac{u^2}{d(x)^2} \, dx + \lambda \int_{\Omega \setminus K} u^2 \, dx \\ &= -\int_{\Omega \setminus K} \Theta(x) u^2 \, dx \geq - |\Theta|_\infty \int_{\Omega \setminus K} u^2 \, dx.
\end{align*}
Hence $\lambda \geq -|\Theta|_\infty$. To show that $\lambda \neq -|\Theta|_\infty$, suppose by contradiction that there is $u \in H^1_0 (\Omega \setminus K)$ with
$$
\int_{\Omega \setminus K} |\nabla u|^2 - \mu \frac{u^2}{|x|^2} - \nu \frac{u^2}{d(x)^2} - \Theta(x) u^2 \, dx = -|\Theta|_\infty \int_{\Omega \setminus K} u^2 \, dx.
$$
Thus
$$
\int_{\Omega \setminus K} |\nabla u|^2 - \mu \frac{u^2}{|x|^2} - \nu \frac{u^2}{d(x)^2} \, dx = \int_{\Omega \setminus K} (\Theta(x) - |\Theta|_\infty) u^2 \, dx \leq 0.
$$
However, using \eqref{hardy1} and \eqref{hardy2} we get
$$
\int_{\Omega \setminus K} |\nabla u|^2 - \mu \frac{u^2}{|x|^2} - \nu \frac{u^2}{d(x)^2} \, dx \geq \left(1 - \frac{4\mu}{(N-2)^2} - 4 \nu \right) \int_{\Omega \setminus K} |\nabla u|^2 \, dx.
$$
Hence $\int_{\Omega \setminus K} |\nabla u|^2 \, dx = 0$ and $u = 0$, which is a contradiction. Hence $\sigma(\cA) \subset (-|\Theta|_\infty, +\infty)$.
\end{proof}

\section{Variational setting and critical point theory}

Suppose that $(E, \| \cdot \|)$ is a Hilbert space and $\cJ : E \rightarrow \R$ is a nonlinear functional of the general form
$$
\cJ(u) = \frac12 \|u\|^2 - \cI(u),
$$
where $\cI$ is of $\cC^1$ class and $\cI(0)=0$. We introduce the so-called \textit{Nehari manifold}
$$
\cN := \{ u \in E \setminus \{ 0 \} \ : \ \cJ'(u)(u) = 0 \}.
$$
Observe that $\cI'(u)(u) > 0$ on $\cN$. Indeed,
$$
0 = \cJ'(u)(u) = \|u\|^2 - \cI'(u)(u), \quad u \in \cN.
$$
To utilize the mountain pass approach, we consider the following space of paths
$$
\Gamma := \{ \gamma \in \cC ([0,1], E) \ : \ \gamma(0) = 0, \ \|\gamma(1)\| > r, \ \cJ(\gamma(1)) < 0 \}
$$
and the following mountain pass level
$$
c := \inf_{\gamma \in \Gamma} \sup_{t \in [0,1]} \cJ(\gamma(t)).
$$
Moreover we set
$$
\Gamma_Q := \Gamma \cap \cC([0,1],Q).
$$
We propose an abstract theorem which is a combination of \cite[Theorem 5.1]{B} and \cite[Theorem 2.1]{BM-Indiana}. The proof is a straightforward modification of proofs of mentioned theorems, however we include it here for the reader's convenience.

\begin{Th}\label{abstract}
Suppose that
\begin{itemize}
\item[(J1)] there is $r > 0$ such that
$$
a := \inf_{\|u\|=r} \cJ(u) > 0;
$$
\item[(J2)] there is a closed vector subspace $Q \subset E$ such that $\frac{\cI (t_n u_n)}{t_n^2} \to +\infty$ for $t_n \to +\infty$, $u_n \in Q$ and $u_n \to u \neq 0$;
\item[(J3)] for all $t > 0$ and $u \in \cN$ there holds
$$
\frac{t^2-1}{2} \cI'(u)(u) - \cI(tu) + \cI(u) \leq 0.
$$
\end{itemize}
Then $\Ga_Q \neq \emptyset$, $\cN \cap Q \neq \emptyset$ and
\begin{equation}\label{gsl}
0 < \inf_{\|u\|=r} \cJ(u) \leq c \leq  \inf_{\gamma \in \Ga_Q} \sup_{t \in [0,1]} \cJ(\gamma(t)) = \inf_{\cN \cap Q} \cJ = \inf_{u \in Q \setminus \{0\}} \sup_{t \geq 0} \cJ(tu),
\end{equation}
Moreover there is a Cerami sequence for $\cJ$ on the level $c$, i.e. a sequence $\{ u_n \}_n \subset E$ such that
$$
\cJ(u_n) \to c, \quad (1+\|u_n\|) \cJ'(u_n) \to 0.
$$
\end{Th}

\begin{proof}
Observe that there exists $v \in Q \setminus \{0\}$ with $\|v\| > r$ such that $\cJ(v) < 0$. Indeed, fix $u \in Q \setminus \{0\}$ and from (J2) there follows that
\begin{equation}\label{infty}
\frac{\cJ(tu)}{t^2} = \frac12 \|u\|^2 - \frac{\cI(tu)}{t^2} \to - \infty \quad \mbox{as } t \to +\infty
\end{equation}
and we may take $v := t u$ for sufficiently large $t > 0$. In particular, the family of paths $\Gamma_Q$ is nonempty. Moreover, $\cJ(tu) \to 0$ as $t \to 0^+$ and for $t = \frac{r}{\|u\|} > 0$ we get $\cJ(tu) > 0$. Hence, taking \eqref{infty} into account, $(0,+\infty) \ni t \mapsto \cJ(tu) \in \R$ has a local maximum, which is a critical point of $\cJ(tu)$ and $tu \in \cN$. Hence $\cN \cap Q \neq \emptyset$. Suppose that $u \in \cN \cap Q$. Then, from (J3),
\begin{equation*}
\cJ(tu) = \cJ(tu) - \frac{t^2-1}{2} \cJ'(u)(u) \leq \cJ(u)
\end{equation*}
and therefore $u$ is a maximizer (not necessarily unique) of $\cJ$ on $\R_+ u := \{ su \ : \ s > 0 \}$. Hence, for any $u \in \cN\cap Q$ there are $0 < t_{\min} (u) \leq 1 \leq t_{\max}(u)$ such that $t u \in \cN\cap Q$ for any $t \in [t_{\min}(u), t_{\max} (u)]$ and 
$$
[t_{\min}(u), t_{\max} (u)] \ni t \mapsto \cJ(tu) \in \R
$$
is constant. Moreover $\cJ'(tu)(u) > 0$ for $t \in (0, t_{\min}(u))$ and $\cJ'(tu)(u) < 0$ for $t \in (t_{\max} (u), +\infty)$, $Q \setminus \cN$ consists of two connected components and any path $\gamma \in \Gamma_Q$ intersects $\cN\cap Q$. Thus
$$
\inf_{\gamma \in \Gamma_Q} \sup_{t \in [0,1]} \cJ(\gamma(t)) \geq \inf_{\cN\cap Q} \cJ.
$$
Since
$$
\inf_{\cN\cap Q} \cJ = \inf_{u \in Q \setminus \{0\}} \sup_{t > 0} \cJ(tu)
$$
there follows, under (J1), that
$$
c = \inf_{\gamma \in \Gamma} \sup_{t \in [0,1]} \cJ(\gamma(t)) \leq \inf_{\gamma \in \Gamma_Q} \sup_{t \in [0,1]} \cJ(\gamma(t)) = \inf_{\cN \cap Q} \cJ = \inf_{u \in Q \setminus \{0\}} \sup_{t > 0} \cJ(tu).
$$
The existence of a Cerami sequence follows from the mountain pass theorem.
\end{proof}

To study the multiplicity of solutions we will recall the symmetric mountain pass theorem. We consider the following condition
\begin{enumerate}
    \item[(J4)] there exists a sequence of subspaces $\widetilde{E}_1 \subset \widetilde{E_2} \subset \ldots \subset E$ such that $\dim \widetilde{E}_k=k$ for every $k \geq 1$ and there is a radius $R_k$ such that $\sup_{u \in \widetilde{E}_k,\ \|u\|\geq R_k} \cJ \leq 0$.
\end{enumerate}
Then, the following theorem holds.

\begin{Th}[{\cite[Corolarry 2.9]{AmbrRab}}, {\cite[Theorem 9.12]{Rabinowitz}}]\label{multipl}
Suppose that $\cJ$, as above, is even and satisfies (J1), (J4) and a Palais-Smale condition (namely, any Palais-Smale sequence for $\cJ$ contains a convergent subsequence). Then $\cJ$ has an unbounded sequence of critical values.
\end{Th}
We work in the usual Sobolev space $H^1_0(\Omega)$ being the completion of $\cC_0^\infty(\Omega)$ with respect to the norm
$$
\| u \|_{H^1} := \left( \int_{\Omega} |\nabla u|^2 + u^2 \, dx \right)^{1/2}.
$$
Define the bilinear form $B : H^1_0(\Omega) \times H^1_0 (\Omega) \rightarrow \R$ by
$$
B(u,v) := \int_{\Omega} \nabla u \cdot \nabla v + \lambda uv \, dx - \mu \int_{\Omega} \frac{uv}{|x|^2} \, dx - \nu \int_\Omega \frac{uv}{d(x)^2} \, dx, \quad u,v \in H^1_0(\Omega).
$$

\begin{Lem}
$B$ defines an inner product on $H^1_0 (\Omega)$. Moreover, the associated norm is equivalent with the usual one.
\end{Lem}
\begin{proof}
To check that $B$ is positive-definite we utilize \eqref{hardy1}, \eqref{hardy2}, and (N) to get
\begin{align*}
B(u,u)& =  \int_{\Omega} |\nabla u |^2 + \lambda u^2 \, dx - \mu \int_{\Omega} \frac{u^2}{|x|^2} \, dx - \nu \int_\Omega \frac{u^2}{d(x)^2} \, dx\\
&\geq \left(1 - \frac{4 \mu}{(N-2)^2} - 4\nu \right) \int_\Omega |\nabla u|^2 \, dx + \lambda u^2 \, dx \geq \left(1 - \frac{4 \mu}{(N-2)^2} - 4\nu \right) \int_\Omega |\nabla u|^2 \, dx
\end{align*}
and the statement follows from the Poincar\'e inequality. Moreover, from
$$
\int_\Omega |\nabla u|^2 + \lambda u^2 \, dx \geq B(u,u) \geq \left(1 - \frac{4 \mu}{(N-2)^2} - 4\nu \right) \int_\Omega |\nabla u|^2 \, dx
$$
there follows that $B$ generates a norm on $H^1_0(\Omega)$ equivalent to the standard one.
\end{proof}
Let $\| \cdot \|$ denote the norm generated by $B$, namely
$$
\|u\| := \sqrt{B(u,u)}, \quad u \in H^1_0(\Omega).
$$
Then we can define the energy functional $\cJ : H^1_0 (\Omega) \rightarrow \R$ by
\begin{equation}\label{eq:J}
\cJ(u) := \frac12 \|u\|^2 - \int_\Omega F(x,u) \, dx,
\end{equation}
where $G(x,u) := \int_0^u g(x,s) \, ds$ and $F$ is given in (F3). It is well-known that under (F1), (F2) the functional is of $\cC^1$ class and
$$
\cJ'(u)(v) = B(u,v) - \int_\Omega f(x,u)v \, dx, \quad u,v \in H^1_0(\Omega).
$$
Hence, its critical points are weak solutions to \eqref{eq}.

\section{Verification of (J1)--(J4)}

Observe that (F1), (F2) imply that for every $\varepsilon > 0$ one can find $C_\varepsilon > 0$ such that
$$
|f(x,u)| \leq \varepsilon |u| + C_{\varepsilon}|u|^{p-1}.
$$
There follows also a similar inequality for $F$, namely
\begin{equation}\label{eq:F-eps}
F(x,u) \leq \varepsilon u^2 + C_\varepsilon |u|^p.
\end{equation}
We note also, that if in addition (F4) holds, then $F(x,u) \geq 0$. Moreover we recall that the functional $\cJ$ is defined by \eqref{eq:J}.

\begin{Lem}\label{lem:assump}
Suppose that (C), (N), (F1)--(F5) hold. Then $\cJ$ satisfies (J1)--(J3) in Theorem \ref{abstract} and (J4) in Theorem \ref{multipl}.
\end{Lem}
\begin{enumerate}
\item[(J1)] Using \eqref{eq:F-eps} and Sobolev embeddings we obtain
$$
\int_\Om F(x,u)\,dx\leq \varepsilon|u|_2^2+C_\varepsilon |u|_p^p\lesssim \varepsilon\|u\|^2+C_\epsilon\|u\|^p.
$$
Hence can choose $\varepsilon>0$ and $r>0$ such that
$$
\int_\Om F(x,u)\,dx\leq \frac{1}{4}\|u\|^2
$$
for all $\|u\|\leq r$. Then we get
\begin{align*}
\mathcal{J}(u) &= \frac{1}{2}\|u\|^2-\int_\Om F(x,u)\,dx \geq \frac{1}{4}\|u\|^2=\frac{r^2}{4}>0
\end{align*}
for all $\|u\|=r.$
\item[(J2)] Let $Q := H^1_0 (\mathrm{int}\, K)$. Let $t_n\to+\infty$, $u_n \in Q$ and $u_n\to u\neq 0$. Then from Fatou's lemma and (F3)
$$
\frac{\mathcal{I}(t_nu_n)}{t_n^2}=\frac{\int_K F(x,t_nu_n)\,dx}{t_n^2} \to+\infty \quad \mbox{as } n \to +\infty.
$$

\item[(J3)] Fix $u \in \cN$. Define
$$
(0,\infty) \ni t \mapsto \varphi(t):=\frac{t^2-1}{2}\mathcal{I}'(u)(u)-\mathcal{I}(tu)+\mathcal{I}(u)\in\mathbb{R}.
$$
Note that $\varphi(1)=0$. Moreover
\begin{align*}
\varphi'(t)&=t \mathcal{I}'(u)(u)-\mathcal{I}'(tu)(u)=\int_\Om f(x,u)tu\,dx-\int_\Om f(x,tu)u\,dx.
\end{align*}
Suppose that $t\in(0,1)$. Then (F3) implies that for a.e. $x \in \Omega$, $f(x,tu)u \leq tf(x,u)u$ and therefore $\varphi'(t) \geq 0$. Similarly $\varphi'(t)\leq 0$ for $t>1$ which implies that $\varphi(t)\leq\varphi(0)=0$ for all $t>0$.
\item[(J4)] Let $\widetilde{E} \subset H^1_0(\mathrm{int}\,K) \subset H^1_0(\Omega)$ be a finite dimensional subspace. Note that on $\widetilde{E}$ all norms are equivalent. Suppose, by contradiction, that there is a sequence $(u_n) \subset \widetilde{E}$ such that $\|u_n\| \to +\infty$ and $\cJ(u_n) > 0$. Let $w_n(x) := u_n(x) / \|u_n\|$. It is clear that $\|w_n\|=1$ and, since $\widetilde{E}$ is finite-dimensional, there is $w \in \widetilde{E} \setminus \{0\}$ such that $\|w_n-w\| \to 0$. In particular $|\supp (w) \cap K| > 0$. Then, for a.e. $x \in \supp (w) \cap K$ we have that
    $$
    u_n(x)^2 = \|u_n\|^2 w_n(x)^2 \to +\infty.
    $$
    Hence, by the Fatou's lemma and (F3)
    $$
    0 < \frac{\cJ(u_n)}{\|u_n\|^2} = \frac12 - \int_{\Omega} \frac{F(x,u_n)}{\|u_n\|^2} \, dx \leq \frac12 - \int_{\supp(w) \cap K} \frac{F(x,u_n)}{u_n^2} w_n^2 \, dx \to -\infty,
    $$
    which is a contradiction.
\end{enumerate}

\section{Cerami sequences and proofs of main theorems}

\begin{Lem}\label{lem:bdd}
Any Cerami sequence for $\cJ$ is bounded.
\end{Lem}

\begin{proof}
Suppose that $\|u_n\|\to+\infty$ up to a subsequence. We define $v_n = \frac{u_n}{\|u_n\|}$. Then $\|v_n\|=1$ and $v_n \weakto v_0$ in $H^1_0(\Omega)$. From compact Sobolev embeddings, $v_n \to v_0$ in $L^2(\Omega)$, in $L^p (\Omega)$ and almost everywhere.

We consider three cases.

\begin{itemize}
\item Suppose that $v_0 = 0$. Condition (J3) implies that
$$
\cJ(u) \geq \cJ(tu) - \frac{t^2-1}{2} \cJ'(u)(u). 
$$
Taking $t\mapsto\frac{t}{\|u_n\|}$ we obtain that 
$$
\cJ(u_n) \geq \cJ\left(\frac{t}{\|u_n\|}u_n\right) - \frac{\frac{t^2}{\|u_n\|^2}-1}{2} \cJ'(u_n)(u_n) = \cJ(t v_n) + o(1).
$$
Hence
$$
\cJ(u_n) \geq \frac{t^2}{2} - \int_\Omega F(tv_n) \, dx + o(1).
$$
Moreover 
$$
\left| \int_\Omega F(tv_n) \, dx \right| \leq \varepsilon t^2 \int_\Omega |v_n|^2 \, dx + C_\varepsilon t^p \int_\Omega |v_n|^p \, dx \to 0 \quad \mbox{as } n \to \infty.
$$
Thus
$$
c + o(1) = \cJ(u_n) \geq \frac{t^2}{2} + o(1)
$$
for any $t > 0$ - a contradiction.
\item Now we suppose that $v_0 \neq 0$ and $|\supp v_0 \cap K | > 0$. Then
$$
o(1)=\frac{\cJ(u_n)}{\|u_n\|^2} = \frac12 - \int_\Omega \frac{F(u_n)}{\|u_n\|^2} \, dx = \frac12 - \int_\Omega \frac{F(u_n)}{u_n^2} v_n^2 \, dx \leq \frac12 - \int_{\supp v_0 \cap K} \frac{F(u_n)}{u_n^2} v_n^2 \, dx \to -\infty,
$$
a contradiction.
\item Suppose that $v_0 \neq 0$ and $|\supp v_0 \cap K| = 0$. Then $\supp v_0 \subset \Omega \setminus K$. Fix $\varphi \in \cC_0^\infty (\Omega \setminus K)$ and note that
$$
o(1) = \cJ'(u_n)(\varphi) = \langle u_n, \varphi \rangle - \int_{\Omega} f(x,u_n) \varphi \, dx.
$$
Observe that 
\begin{align*}
    \int_{\Omega} f(x,u_n)\varphi \, dx = \|u_n\| \int_{\Omega} \frac{f(x,u_n)}{u_n} v_n \varphi \, dx = \|u_n\| \left( \int_{\supp \varphi \cap \supp v} \frac{f(x,u_n)}{u_n} v_n \varphi \, dx + o(1) \right)
\end{align*}
Observe that for a.e. $x \in \supp \varphi \cap \supp v$ we get that $|u_n(x)| = |v_n(x)| \|u_n\| \to +\infty$. Fix $x \in \supp \varphi \cap \supp v$; then from (F5), for sufficiently large $n$, $f(x,u_n(x)) = \Theta(x) u_n(x)$. Thus
$$
\frac{f(x,u_n(x))}{u_n(x)} v_n(x) \varphi(x) = \Theta(x) v_n(x) \varphi(x)
$$
and therefore
$$
\frac{f(x,u_n(x))}{u_n(x)} v_n(x) \varphi(x) \to \Theta(x) v_n(x) \varphi(x)
$$
pointwise, a.e. on $\supp \varphi \cap \supp v$. Combining (F4) and (F5) we also get that
$$
\left| \frac{f(x,u_n(x))}{u_n(x)}  \right|^2 \leq |\Theta|_\infty^2
$$
and $\frac{f(\cdot ,u_n)}{u_n} \to \Theta$ in $L^2(\supp \varphi \cap \supp v)$. Thus, from Lebesgue dominated convergence theorem and the H\"older inequality we get
$$
\int_{\supp \varphi \cap \supp v} \frac{f(x,u_n)}{u_n} v_n \varphi \, dx \to \int_{\Omega} \Theta(x) v \varphi \, dx.
$$
Hence
$$
\langle v, \varphi \rangle = \int_\Omega \Theta(x) v \varphi \, dx.
$$
In particular, $0$ is an eigenvalue of the operator $-\Delta + \lambda - \frac{\mu}{|x|^2} - \frac{\nu}{d^2} - \Theta(x)$ with Dirichlet boundary conditions on $\Omega \setminus K$, which contradicts (A).
\end{itemize}

\end{proof}

\begin{proof}[Proof of Theorem \ref{th:main1}]
Since Cerami sequence $u_n$ is bounded we have following convergences (up to a subsequence):
\begin{align*}
u_n\weakto u_0 \quad & \mbox{in } H^1_0(\Omega),\\
u_n\to u_0 \quad & \mbox{in } L^2(\Om), \mbox{ and in } L^p(\Om),\\
u_n\to u_0\quad & \mbox{a.e. on }\Om.
\end{align*}
Hence, for any $\varphi\in \cC^\infty_0(\Omega)$,
$$
\cJ'(u_n)(\varphi)-\cJ'(u_0)(\varphi)=\langle u_n-u_0,\varphi\rangle-\int_\Om \left(f(x,u_n)-f(x,u_0)\right)\varphi\,dx\to 0,
$$
because obviously weak convergence of $u_n$ implies that
$$
\langle u_n-u_0,\varphi\rangle\to 0,
$$
and we will use the Vitali convergence theorem to prove that
$$ 
\int_\Om \left(f(x,u_n)-f(x,u_0)\right)\varphi\,dx\to 0.
$$
Hence we need to check the uniform integrability of the family $\left\{ \left(f(x,u_n)-f(x,u_0)\right)\varphi \right\}_n$. Using (F1) and Lemma \ref{lem:bdd} we obtain that for any measurable set $E \subset \Omega$
\begin{align*}
\int_E|f(x,u_n)-f(x,u_0)|\varphi\,dx &\leq  \int_E|f(x,u_n)\varphi|\,dx+\int_E|f(x,u_0)\varphi|\,dx\\
&\lesssim  \int_E|\varphi|\,dx+\int_E|u_n|^{p-1}|\varphi|\,dx+\int_E|\varphi|\,dx+\int_E|u_0|^{p-1}|\varphi|\,dx\\
&\lesssim  |\varphi \chi_E|_1+ |u_n|_p^{p-1} |\varphi \chi_E|_p  +|u_0|_p^{p-1} |\varphi \chi_E|_p \\
&\lesssim |\varphi \chi_E|_1+|\varphi \chi_E|_p.
\end{align*}
Then, for any $\varepsilon > 0$, we can choose $\delta > 0$ small enough that
$$
\int_E|f(x,u_n)-f(x,u_0)|\varphi\,dx < \varepsilon
$$
for $|E|<\delta$. Hence $\cJ'(u_n)(\varphi)\to\cJ'(u_0)(\varphi)$, and $\cJ'(u_0)=0$.
\end{proof}

\begin{proof}[Proof of Theorem \ref{th:main2}]
The statement follows directly from Theorem \ref{multipl} and Lemma \ref{lem:assump}.
\end{proof}

\section{Multiple solutions to the mass-subcritical normalized problem}

In what follows we are interested in the normalized problem \eqref{eq:normalized}, where $\lambda$ is not prescribed anymore and is the part of the unknown $(\lambda, u) \in \R \times H^1_0(\Omega)$. Then, solutions are critical point of the energy functional 
$$
\cJ_0 (u) := \frac12 |\nabla u|_2^2 - \frac{\mu}{2} \int_\Omega \frac{u^2}{|x|^2} \, dx - \frac{\nu}{2} \int_\Omega \frac{u^2}{d(x)^2} \, dx - \int_\Omega F(x,u) \, dx
$$
restricted to the $L^2$-sphere in $H^1_0(\Omega)$
$$
\cS := \left\{ u \in H^1_0(\Omega) \ : \ \int_\Omega u^2 \, dx = \rho \right\}
$$
and $\lambda$ arises as a Lagrange multiplier.

We recall the well-known Gagliardo-Nirenberg inequality
\begin{equation}\label{gn-ineq}
|u|_p \leq C_{p,N} |\nabla u|_2^{\delta_p} |u|_2^{1-\delta_p}, \quad u \in H^1_0(\Omega),
\end{equation}
where $\delta_p := N \left(\frac12 - \frac{1}{p} \right)$ and $C_{p,N} > 0$ is the optimal constant.

\begin{Lem}\label{coercive}
$\cJ_0$ is coercive and bounded from below on $\cS$.
\end{Lem}

\begin{proof}
Using (F1), \eqref{hardy1}, \eqref{hardy2} and \eqref{gn-ineq}, we obtain
\begin{align*}
\cJ_0(u) &=  \frac12 |\nabla u|_2^2 - \frac{\mu}{2} \int_\Omega \frac{u^2}{|x|^2} \, dx - \frac{\nu}{2} \int_\Omega \frac{u^2}{d(x)^2} \, dx - \int_\Omega F(x,u) \, dx  \\
&\geq \frac12 \left(1 - \frac{4\mu}{(N-2)^2} - 4\nu \right) |\nabla u|_2^2 -C_1 |\Om| - C_1|u|_p^p \\
&\geq \frac12 \left(1 - \frac{4\mu}{(N-2)^2} - 4\nu \right) |\nabla u|_2^2 - C_1|\Om|- C\left(|\nabla u|_2^{\delta_p} |u|_2^{1-\delta_p} \right)^p\\
&\geq \frac12 \left(1 - \frac{4\mu}{(N-2)^2} - 4\nu \right) |\nabla u|_2^2 -C_1|\Om|+ C|\nabla u|_2^{\delta_pp},
\end{align*}
where 
$$
\delta_p p = N \left( \frac12 - \frac1p \right) p = N \left( \frac{p}{2}-1 \right) < N \cdot \frac{2}{N} = 2.
$$
Thus $\cJ_0$ is coercive and bounded from below on $\cS$.
\end{proof}

\begin{Lem}
$\cJ_0$ satisfies the Palais-Smale condition on $\cS$, i.e. any Palais-Smale sequence for $\cJ_0 |_\cS$ has a convergent subsequence.
\end{Lem}

\begin{proof}
Let $(u_n) \subset \cS$ be a Palais-Smale sequence for $\cJ_0 |_\cS$. Then Lemma \ref{coercive} implies that $(u_n)$ is bounded in $H_0^1 (\Omega)$. Hence we may assume that (up to a subsequence)
\begin{align*}
u_n \rightharpoonup u \quad & \mbox{in } H^1_0(\Omega), \\
u_n \to u \quad & \mbox{in } L^p (\Omega), \\
u_n \to u \quad & \mbox{a.e. on } \Omega.
\end{align*}
Moreover
$$
\cJ_0'(u_n) + \lambda_n u_n \to 0 \quad \mbox{in } H^{-1}(\Omega) := (H^1_0 (\Omega))^*
$$
for some $\lambda_n \in \R$. In particular
$$
\cJ_0'(u_n)(u_n) + \lambda_n |u_n|_2^2 \to 0.
$$
Note that
$$
\lambda_n = - \frac{\cJ_0'(u_n)(u_n)}{|u_n|_2^2}  + o(1)= - \frac{\|u_n\|^2 - \int_\Omega f(x,u_n)u_n \, dx}{|u_n|_2^2} + o(1).
$$
Observe that, from (F1)
$$
\left| \int_\Omega f(x,u_n)u_n \, dx \right| \lesssim 1 + |u_n|_p^p \lesssim 1.
$$
Therefore $(\lambda_n) \subset \R$ is bounded, and (up to a subsequence) $\lambda_n \to \lambda_0$. Therefore, up to a subsequence,
\begin{align*}
    o(1) &= \cJ_0'(u_n)(u_n) + \lambda_n |u_n|_2^2 - \cJ_0'(u_n)(u) - \lambda_n \int_{\Omega} u_n u \, dx \\
    &= \cJ_0'(u_n)(u_n) - \cJ_0'(u_n)(u) + \lambda_n \int_{\Omega} u_n (u_n-u) \, dx \\
    &= \|u_n\|^2 - \langle u_n, u \rangle - \int_{\Omega} f(x,u_n) (u_n-u) \, dx +  \lambda_n \int_\Omega u_n (u_n-u) \, dx.
\end{align*}
It is clear that $\langle u_n, u\rangle \to \|u\|^2$ and that 
$$
\left| \lambda_n \int_\Omega u_n (u_n-u) \, dx \right| \lesssim |u_n|_2^2 |u_n-u|_2^2 \to 0.
$$
Moreover, from (F1)
\begin{align*}
    \left| \int_\Omega f(x,u_n)(u_n-u) \, dx \right| \lesssim |u_n-u|_2 + |u_n|_p^{p-1} |u_n-u|_p \to 0.
\end{align*}
Hence $\|u_n\| \to \|u\|$ and therefore $u_n \to u$ in $H^1_0(\Omega)$.
\end{proof}

\begin{proof}[Proof of Theorem \ref{main:3}]
From \cite[Theorem II.5.7]{Str} we obtain that $\cJ_0$ has at least $\hat{\gamma}(\cS)$ critical points, where
$$
\hat{\gamma} (\cS) := \sup \{ \gamma(K) \ : \ K \subset \cS \mbox{ - symmetric and compact} \}
$$
and $\gamma$ denotes the Krasnoselskii's genus for symmetric and compact sets. We will show that $\hat{\gamma}(\cS) = +\infty$. Indeed, fix $k \in \mathbb{N}$. It is sufficient to construct a symmetric and compact set $K \subset \cS$ with $\gamma(K) = k$. Choose functions $w_1, w_2, \ldots, w_k \in \cC_0^\infty (\Omega) \cap \cS$ with pairwise disjoint supports, namely $w_i w_j = 0$ for $i \neq j$. Now we set
$$
K := \left\{ \sum_{i=1}^k \alpha_i w_i \in \cS \ : \  \sum_{i=1}^k \alpha_i^2 = 1 \right\}.
$$
It is clear that $K \subset \cS$ is symmetric and compact. We will show that $\gamma(K) = k$. In what follows $S^{m-1}$ denotes the $(m-1)$-dimensional sphere in $\R^m$ of radius $1$ centered at the origin. Note that $h : K \rightarrow S^{k-1}$ given by
$$
K \ni \sum_{i=1}^k \alpha_i w_i \mapsto h \left( \sum_{i=1}^k \alpha_i w_i \right) := (\alpha_1, \ldots, \alpha_k) \in S^{k-1}
$$
is a homeomorphism, which is odd. Hence $\gamma(K) \leq k$. Suppose by contradiction that $\gamma(K) < k$. Then there is a continuous and odd function $\widetilde{h} : K \rightarrow S^{\gamma(K) -1}$. However, $\widetilde{h} \circ h^{-1} : S^{k-1} \rightarrow S^{\gamma(K)-1}$ is an odd, continuous map, which contradicts the Borsuk-Ulam theorem \cite[Proposition II.5.2]{Str}, \cite[Theorem D.17]{Willem}. Hence $\gamma(K) = k$.
\end{proof}

\section*{Acknowledgements}

Adam Konysz was partly supported by the Excellence Initiative - Research University at Nicolaus Copernicus University (Grants4NCUStudents 90-SIDUB.6102.18.2021.G4NCUS2).  Moreover, some parts of this work were completed during his short research stays at Institute of Mathematics, Polish Academy of Sciences; he wishes to thank for the warm hospitality. Bartosz Bieganowski and Adam Konysz were also partly supported by National Science Centre, Poland (Grant No. 2017/26/E/ST1/00817).


\begin{thebibliography}{99}


\bibitem{AmbrRab} A. Ambrosetti, P.H. Rabinowitz: \textit{Dual variational methods in critical point theory and applications}, J. Funct. Anal. \textbf{14}, Issue 4 (1973), 349--381.

\bibitem{BFT} G. Barbatis, S. Filippas, A. Tertikas: \textit{A unified approach to improved Lp Hardy inequalities with best constants}, Trans. Amer.Math. Soc. \textbf{356}, no. 6 (2003), 2169--2196.

\bibitem{B} B. Bieganowski: \textit{Solutions to a nonlinear Maxwell equation with two competing nonlinearities in $\mathbb{R}^3$}, Bulletin Polish Acad. Sci. Math. \textbf{69} (2021), p. 37--60.

\bibitem{B2} B. Bieganowski: \textit{Systems of coupled Schrödinger equations with sign-changing nonlinearities via classical Nehari manifold approach}, Complex Var. Elliptic Equ. \textbf{64}, Issue 7 (2019), p. 1237--1256.

\bibitem{BM-Indiana} B. Bieganowski, J. Mederski: \textit{Bound states for the Schr\"{o}dinger equation with mixed-type nonlinearites}, Indiana Univ. Math. J. \textbf{71} No. 1 (2022), p. 65--92.

\bibitem{BV} H. Brezis, , J.-L. V\'{a}zquez: \textit{Blow-up solutions of some nonlinear elliptic problems}, Rev. Mat. Univ. Compl. Madrid \textbf{10}, no. 2 (1997), p. 443--469.

\bibitem{Chen} H. Chen, L. V\'{e}ron: \textit{Schr\"odinger operators with Leray-Hardy potential singular on the boundary}, J. Differential Equations \textbf{269}, Issue 3 (2020), p. 2091--2131.

\bibitem{Dorfler} W. D\"{o}rfler, A. Lechleiter, M. Plum, G. Schneider, C. Wieners: \textit{Photonic Crystals: Mathematical Analysis and Numerical Approximation}, Springer Basel 2012.

\bibitem{Ev} L.C. Evans, R.F. Gariepy: \textit{Measure theory and fine properties of functions}, CRC Press 1992.

\bibitem{Felli} V. Felli, S. Terracini: \textit{Nonlinear Schr\"odinger equations
with symmetric multi-polar potentials}, Calc. Var. Partial Differential Equations \textbf{27}(1), (2006), p. 25--58.

\bibitem{Fibich} G. Fibich, F. Merle: \textit{Self-focusing on bounded domains}, Phys. D, \textbf{155}(1-2), (2001), p. 132--158.

\bibitem{Fibich2} G. Fibich, D. Shpigelman: \textit{Positive and necklace solitary waves on bounded domains}, Phys. D \textbf{315} (2016), p. 13--32.

\bibitem{Gao} F. Gao, Y. Guo: \textit{Multiple solutions for quasilinear equation involving Hardy critical Sobolev exponents}, Topol. Methods Nonlinear Anal. \textbf{56}, No. 1 (2020), p. 31--61.

\bibitem{GuoMed} Q. Guo, J. Mederski: \textit{Ground states of nonlinear Schrödinger equations with sum of periodic and inverse-square potentials}, J. Differential Equations \textbf{260} (2016), no. 5, p. 4180--4202.

\bibitem{Kostenko}
A.Kostenko, A. Sakhnovich, G. Teschl: \textit{Commutation methods for Schr\"odinger operators with strongly singular potentials}, Math. Nachr. \textbf{285}, Issue 4 (2012), p. 392--410.

\bibitem{Tai} M. Marcus, P.-T. Nguyen:  \textit{Schr\"odinger equations with singular potentials: linear
and nonlinear boundary value problems}, Math. Ann. \textbf{374} (2019), p. 361--394.

\bibitem{Tai2} M. Marcus, P.-T. Nguyen: \textit{Moderate solutions of semilinear elliptic equations with Hardy potential}, Ann. Inst. H. Poincar\'{e} (C) Non Linear Anal. \textbf{34} (2015), p. 69--88.

\bibitem{Nie} W. Nie: \textit{Optical Nonlinearity: Phenomena, applications, and materials}, Adv. Mater. \textbf{5} (1993), p. 520--545.

\bibitem{Noris} B. Noris, H. Tavares, G. Verzini: \textit{Existence and orbital stability of the ground states with prescribed mass for the $L^2$-critical and supercritical NLS on bounded domains}, Anal. PDE \textbf{7}(8), (2014), p. 1807--1838.

\bibitem{PV} D. Pierotti, G. Verzini: \textit{Normalized bound states for the nonlinear Schr\"odinger equation in bounded domains}, Calc. Var. Partial Differential Equations \textbf{56}, Article number: 133 (2017).

\bibitem{Rabinowitz} P.H. Rabinowitz: \textit{Minimax methods in critical point theory with applications to differential equations}, CBMS Regional Conference Series in Mathematics, 65, American Mathematical Society, Providence, RI (1986).

\bibitem{Str} M. Struwe: \textit{Variational methods}, Springer 2008.

\bibitem{Zuazua} J. Vancostenoble, E. Zuazua: \textit{Hardy Inequalities, Observability, and Control for the Wave and Schrödinger Equations with Singular Potentials}, SIAM J. Math. Anal. \textbf{41}, Issue 4 (2009), p. 1508--1532.

\bibitem{Willem} M. Willem: \textit{Minimax Theorems}, Birkh\"{a}user Verlag, 1996.

\bibitem{Ziemer} W.P. Ziemer: \textit{Weakly Differentiable Functions}. Sobolev Spaces and Functions of Bounded Variation, Graduate Texts in Mathematics, Vol. \textbf{120}, Springer-Verlag, New York, 1989.

\end{thebibliography}
\end{document}